\documentclass[a4paper,10pt]{article}

\usepackage{graphicx}
\usepackage{geometry}
\usepackage{amsthm}
\usepackage{amssymb}
\usepackage{amsmath}
\usepackage{hyperref}
\usepackage{xcolor}
\usepackage{enumerate}
\usepackage{listings}
\usepackage{complexity}
\usepackage{blkarray}

\usepackage[font=small,labelfont=bf]{caption}
\usepackage[font=small,labelfont=normalfont,labelformat=simple]{subcaption}

\newcommand{\adi}{{\rm adi}}
\renewcommand{\deg}{{\rm deg}}
\newcommand{\ava}{{\rm ava}}
\newcommand{\va}{{\rm va}}
\newcommand{\dac}{{\rm dac}}
\graphicspath{{figs/}}

\newtheorem{theorem}{Theorem}
\newtheorem{definition}{Definition}
\newtheorem{lemma}[theorem]{Lemma}
\newtheorem{corollary}[theorem]{Corollary}

\newtheorem{observation}{Observation}
\newtheorem{proposition}{Proposition}

\newcommand*\samethanks[1][\value{footnote}]{\footnotemark[#1]}

\author
{
Stefan Felsner\thanks{Institut f\"ur Mathematik, Technische Universit\"at Berlin, Germany, \hfill{}\goodbreak
\texttt{\{felsner,steiner\}@math.tu-berlin.de}}
\and
Winfried Hochst\"{a}ttler \thanks{Fakult\"{a}t f\"ur Mathematik und Informatik, FernUniversit\"{a}t in Hagen, Germany, \hfill{}\goodbreak
\texttt{winfried.hochstaettler@fernuni-hagen.de}}
\and
Kolja Knauer \thanks{Université Aix-Marseille
Laboratoire d'Informatique et des Systèmes (LIS)
Algorithmique, Combinatoire et Recherche Opérationnelle (ACRO), France \hfill{}\goodbreak
\texttt{kolja.knauer@lis-lab.fr}}
\and
Raphael Steiner\samethanks[1]{}
}

\date{\today}

\title{Complete Acyclic Colorings}

\begin{document}
\maketitle

\begin{abstract}
We study two parameters that arise from the dichromatic number and the vertex-arboricity in the same way that the achromatic number comes from the chromatic number.
The \emph{adichromatic number} of a digraph is the largest number of colors its vertices can be colored with such that every color induces an acyclic subdigraph but merging any two colors yields a monochromatic directed cycle. Similarly, the \emph{a-vertex arboricity} of an undirected graph is the largest number of colors that can be used such that every color induces a forest but merging any two yields a monochromatic cycle. 
We study the relation between these parameters and their behavior with respect to other classical parameters such as degeneracy and most importantly feedback vertex sets.
\end{abstract}

\section{Introduction}
All digraphs and graphs in this paper are considered loopless. For digraphs, we allow parallel and anti-parallel edges, graphs may have multiple edges. An anti-parallel pair of edges in a digraph will be called a \emph{digon} and treated as a directed cycle of length $2$, while a parallel pair of edges in an undirected graph will be called a \emph{bigon} and treated as an undirected cycle of length $2$. The character of our problems often depends on the existence of digons and bigons, respectively. We will emphasize this distinction at the respective points.

A \emph{complete coloring} of a graph is a proper vertex coloring such that the identification of any two colors produces a monochromatic edge. The \emph{achromatic number} $\Psi(G)$ is the {maximum} number of colors in a complete coloring. There has been a substantial amount of research on the achromatic number since its introduction in~\cite{achromatic}, we refer to~\cite{survey1} and~\cite{survey2} for survey articles on this topic.

In the same spirit, for most coloring parameters an associated notion of \emph{complete coloring} and an 'a-parameter' may be defined.
The \emph{dichromatic number} of directed graphs has been studied as a natural directed analogue of the chromatic number of graphs. One of the main interest is to compare its properties to the undirected theory, see~\cite{HARUTYUNYAN2019,dig4,fractionalNL} for some recent results. A coloring in this setting consists of a partition of the vertex set into subsets inducing acyclic subdigraphs. 

Similarly, for an undirected graph, the \emph{vertex arboricity} is defined to be the minimal number of induced forests which cover all the vertices. This is another well-studied parameter, see e.g. \cite{Chartrand1968}, \cite{10.1112/jlms/s1-44.1.612}. Note that if a digraph has no digons, then its dichromatic number is at most the vertex-arboricity of the underlying undirected graph, while  the dichromatic number of a bidirected digraph coincides with the chromatic number of the underlying undirected graph.

In this paper, we investigate complete colorings corresponding to the two above coloring parameters, resulting in the {adichromatic number} of directed graphs and the {a-vertex arboricity} of undirected graphs. More precisely, the \emph{adichromatic number} $\adi(D)$ of a directed graph $D$ is the largest number of colors its vertices can be colored with such that every color induces an acyclic subdigraph but in the merge of any two color classes there is a directed cycle. We refer to such a coloring as a \emph{complete (acyclic) coloring} of $D$. Similarly, the \emph{a-vertex arboricity} $\ava(G)$ of an undirected graph $G$ is the largest number of colors that can be used such that every color induces a forest but in the merge of any two color classes there is a cycle. Such a coloring will be referred as a \emph{complete (arboreal) coloring} of $G$.

Similar parameters have been introduced in~\cite{harmonious,oriented}. In particular, the \emph{diachromatic number} was introduced in~\cite{diachromatic} and sparked our investigations. While it is closest to our parameter it still behaves quite differently (Proposition~\ref{farfromtight}).
Here, we initiate the study of $\adi$ and $\ava$, their relation to other graph parameters and their behavior with respect to graph and digraph operations. In particular, we obtain an interpolation theorem, similar statements have been proven for other parameters (Theorem~\ref{interpolation}).

Our main results concern 
the relation of $\adi$ and $\ava$ to important graph and digraph parameters such as the degeneracy and more importantly the size $\tau$ of a smallest feedback vertex set. 
While it is easy to see that both $\adi$ and $\ava$ are bounded from above by $\tau+1$ (Proposition~\ref{FVupperbound}), there is no function $f$ such that $\tau\leq f(\adi)$ in the directed setting (Proposition~\ref{nobound}). On the other hand we show that there is an $f$ such that $\tau\leq f(\ava)$ in the undirected setting for simple graphs (Theorem~\ref{FVlowerbound}).
We complement the negative results for digraphs by showing that for orientations of members of any fixed non-trivial minor closed class of graphs (for instance for simple planar digraphs), there is an $f$ such that $\tau\leq f(\adi)$ (Theorem \ref{minor-closed}).

A corollary of the above is that there is an $f$ such that $\adi(D)\leq f(\ava(G))$ for any orientation $D$ of a simple graph $G$ (Corollary~\ref{interplay}). We moreover show that any graph with sufficiently high a-vertex arboricity has an orientation with large adichromatic number (Proposition~\ref{outerplay}).

In a similar vein, we show that there are arbitrarily dense digraphs with adichromatic number $1$, but the degeneracy of a simple graph is upper bounded by a function of $\ava$ (Theorem~\ref{degeneracy}). 

From a more general point of view, we discuss the relations of our findings to the Erd\H{o}s-P\'osa property and introduce a strengthening of the latter that is called $\tau$-boundedness.

\paragraph{Structure of the paper:}
Section~\ref{sec:prelim} contains first observations and results. It consists of Subsection~\ref{subsec:param}, which studies some relations between the parameters and contains the interpolation theorem, and of Subsection~\ref{subsec:operations}, which contains results on the behavior of our parameters with respect to graph and digraph operations.
Section~\ref{main} contains our main results, i.e., the above mentioned relations of $\ava$ and $\adi$ with $\tau$ and the degeneracy are given.
Finally, Section~\ref{sec:discuss} contains the discussion of $\tau$-boundedness and the Erd\H{o}s-P\'osa property in Subsection~\ref{subsec:erdosposa} as well as properties with respect to randomness in Subsection~\ref{subsec:random}.

\section{First observations and results}\label{sec:prelim}
In this section, we present some basic properties of the adichromatic number and the a-vertex arboricity, discuss the relation of the adichromatic number and some other notions of complete coloring and make observations which will be used frequently in the rest of the paper.

\subsection{Relations between parameters}\label{subsec:param}

A fundamental property of an optimal proper coloring of a digraph (i.e., without monochromatic directed cycles) is that we cannot improve the coloring by any merge of two colors. Therefore, any coloring of a digraph $D$ with $\vec{\chi}(D)$ colors must be complete. With the analogous argument for the a-vertex arboricity, we get the following basic relations.
\begin{observation} \label{divsadi}
\noindent
\begin{itemize}
\item $\vec{\chi}(D) \leq \adi(D)$ for any digraph $D$.
\item $\va(G) \leq \ava(G)$ for any graph $G$.
\end{itemize}
\end{observation}
Clearly, the only digraphs with adichromatic number $1$ are DAGs, and similarly the only graphs with a-vertex arboricity $1$ are forests. The same is true for the dichromatic number and the vertex arboricity, and so in these very special cases, the concepts of coloring and complete coloring coincide. The examples from Proposition \ref{examples} below show that in general, no such relation exists, as there are digraphs with dichromatic number $2$ and unbounded adichromatic number and graphs with vertex arboricity $2$ but unbounded a-vertex arboricity.

A \emph{feedback vertex set} in a digraph $D$ is a set $F \subseteq V(D)$ whose deletion results in an acyclic digraph, that is, every directed cycle uses a vertex from $F$. Similarly, a feedback vertex set in a graph $G$ is a set of vertices whose deletion results in a forest, i.e., every cycle uses a vertex from the feedback set. The minimal size of a feedback vertex set in a digraph or graph will be denoted by $\tau(D)$ respectively $\tau(G)$. 
%
We start with the following fundamental relation of the treated coloring parameters to feedback vertex sets.
\begin{proposition} \label{FVupperbound}
\noindent
\begin{itemize}
\item For any digraph $D$, we have $\adi(D) \leq \tau(D)+1$.
\item For any graph $G$, we have $\ava(G) \leq \tau(G)+1$.
\end{itemize}
\end{proposition}
\begin{proof}
We give the proof for the directed case, the undirected case is completely analogous. Denote by $F \subseteq V(D)$ a directed feedback vertex set of $D$ with minimal size. Assume towards a contradiction that $\adi(D) \ge |F|+2$. Then there is a complete acyclic partition $(V_1,\ldots,V_k)$ of $D$ with $k \ge |F|+2$ colors. Consequently, there are at least two colors $i,j \in \{1,\ldots,k\}$ which do not appear on any vertex of $F$. However, this implies that $V_i \cup V_j \subseteq V(D) \setminus F$ induces an acyclic subdigraph of $D$, contradicting the definition of a complete acyclic coloring.
\end{proof}
The given upper bounds are easily seen to be tight for acyclic digraphs, directed cycles of arbitrary length, and complete digraphs (obtained from a complete graph by replacing each edge by a digon) in the directed case, respectively for forests, cycles and multi-graphs obtained from complete graphs by replacing each simple edge by a bigon in the undirected case. It is a natural question whether there is also an inverse relationship between the parameters $\adi(D)$ and $\tau(D)$ respectively $\ava(G)$ and $\tau(G)$. This question is central to our paper and will be dealt with in Section \ref{main}.

Since digons and bigons are counted as (directed) cycles of length two, it is easily seen that both the adichromatic number and the achromatic number form a proper generalization of the achromatic numbers of graphs when allowing multiple edges.
\begin{observation} \label{achrom}
For a simple graph $G$, let $\overset{\leftrightarrow}{G}$ denote the bidirected graph where every edge of $G$ is replaced by a digon, and let $2G$ be the multi-graph obtained from $G$ by doubling the edges. Then we have the following equalities involving the achromatic number of $G$:
$$\Psi(G)=\adi(\overset{\leftrightarrow}{G})=\ava(2G).$$
\end{observation}
\begin{proof}
The three parameters are each defined as the maximal size of a partition into independent respectively acyclic vertex sets of maximal size such that the union of any two partition classes is not independent respectively acyclic any more. The claim therefore follows from observing that a vertex set in $\overset{\leftrightarrow}{G}$ induces an acyclic subdigraph if and only if it induces a forest in $2G$ and if and only if it is independent in $G$.
\end{proof}

As independent sets form a special case of acyclic vertex sets in graphs, which again define acyclic vertex sets in any orientation of that graph, it is easily seen that for any digraph $D$ with underlying graph $G$, we have $\vec{\chi}(D) \leq \va(G) \leq \chi(G)$. It is therefore natural to ask whether similar relationships between the adichromatic number, a-vertex arboricity and achromatic number exist. 
The following presents a set of canonical graphs and digraphs with their $a$-coloring parameters, which show that in general, both $\adi(D)$ and $\ava(G)$ are not upper bounded in terms of $\Psi(G)$, where $G$ is the underlying graph of $D$. Note that vice-versa, $\Psi(G)$ cannot be upper bounded in terms of $\adi(D)$ or $\ava(G)$, as it is already unbounded on matchings.
\begin{proposition}\label{examples}
For any $m, n \in \mathbb{N}$, $m, n \ge 1$ we have
\begin{enumerate}
\item $\ava(K_{m,n})=\min\{m,n\}$, while $\Psi(K_{m,n})=2$.
\item $\ava(K_n)=\lceil \frac{n}{2} \rceil$.
\item Let $D_{n}$ be the orientation of $K_{n,n}$ in which a perfect matching is directed from the first to the second class of the bipartition while all non-matching edges emanate from the second class. Then $\adi(D_n)=n$.
\item Let $G$ respectively $D$ be the vertex-disjoint union of $\binom{n}{2}$ cycles respectively directed cycles, $n \ge 2$. Then $\ava(G)=\adi(D)=n.$
\end{enumerate}
\end{proposition}
\begin{proof}
\noindent
\begin{enumerate}
\item Assume that $m \le n$. We first observe that $\ava(K_{m,m}) \ge m$: Consider a perfect matching and assign $m$ different colors to the pairs of matched vertices. Now the union of any two color classes produces a cycle of length $4$. Because $K_{m,m}$ is an induced subgraph of $K_{m,n}$, it is not hard to see that $\ava(K_{m,n}) \ge \ava(K_{m,m}) \ge m=\min\{m,n\}$ (this will be noted later in Corollary \ref{InducedMonotonicity}). On the other hand, removing all but $1$ vertex from the smaller partite class of $K_{m,n}$ shows that $\tau(K_{m,n}) \leq \min\{m,n\}-1$. The claim follows from Proposition \ref{FVupperbound}. 
\item  In any complete arboreal coloring of $K_n$, each color class has size at most $2$ and there is at most one singleton-color class (otherwise two singleton-colors could be merged). We therefore have $\ava(K_n) \leq \frac{n}{2}$ if $n$ is even and $\ava(K_n) \leq \frac{n}{2}+1$ if $n$ is odd. 

If $n$ is even, $K_n$ contains $K_{n/2,n/2}$ as an induced subgraph and so $\ava(K_n) \ge \ava(K_{n/2,n/2})=\frac{n}{2}$. If $n$ is odd, we can use $\frac{n-1}{2}$ colors to color paired vertices of $K_{n-1}$ and then add an extra color for the remaining vertex. This yields a complete arboreal coloring with $\frac{n}{2}+1$ colors. This verifies the claim in both the even and odd case.
\item To verify $\adi(D_n) \ge n$, we observe that by taking the two vertices of each matching edge as a color class we define a partition into $n$ acyclic subdigraphs such that the union of any two creates a directed cycle of length $4$. Deleting all vertices of one partite class except one shows that $\tau(D_n) \leq n-1$. Again, the claim follows using Proposition \ref{FVupperbound}.
\item In any complete arboreal respectively acyclic coloring of $G$ respectively $D$, we must have a cycle (respectively directed cycle) in the union of any two color classes, and so if we use $k$ colors, we have at least $\binom{k}{2}$ distinct cycles. This proves $\ava(G),\adi(D) \leq n$. On the other hand, assigning to each pair $\{i,j\} \in \binom{[n]}{2}$ a different cycle and coloring this cycle using only $i$ and $j$ defines a complete arboreal (respectively acyclic) coloring of $G$ (respectively $D$) and proves that $\ava(G), \adi(D) \ge n$.
\end{enumerate}
\end{proof}

Recently, the concept of the {diachromatic number} was introduced in~\cite{diachromatic}. Given a digraph $D$, the \emph{diachromatic number} $\dac(D)$ of $D$ is defined as the maximum number $k$ of colors that can be used in a vertex-coloring of $D$ with acyclic partition classes $V_1,\ldots,V_k$, such that for any $(i,j) \in [k]^2$, $i \neq j$, there exists at least one arc of $D$ with head in $V_i$ and tail in $V_j$. 

In a complete acyclic coloring of a digraph, the union of any two color classes contains the vertex set of a directed cycle, and therefore arcs in both directions between the two acyclic color classes must exist. Hence, any complete acyclic coloring in our sense also defines a complete coloring as defined in \cite{diachromatic}. We therefore have
\begin{observation}\label{adileqdia}
For any digraph $D$, it holds that $\adi(D) \leq \dac(D).$
\end{observation}

In general however, the above estimate is far from being tight. For example, the diachromatic number is not bounded for directed acyclic digraphs, which always have adichromatic number equal to $1$. Together with  \cite[Corollary 12]{diachromatic} we get:
\begin{proposition}\label{farfromtight}
Let $T_n$ denote the transitive tournament on $n$ vertices. It holds that $$\adi(T_n)=1, \dac(T_n)=\left\lceil \frac{n}{2} \right\rceil.$$
\end{proposition}

%
%
%

For other notions of complete colorings, so-called \emph{interpolation theorems} have been shown. See~\cite[Theorem 22]{diachromatic} and~\cite{achromatic} for the diachromatic and achromatic versions, respectively. Here we extend these results to the adichromatic number and the a-vertex arboricity.

%

\begin{theorem}\label{interpolation}
\noindent
\begin{enumerate}
\item Let $D$ be a digraph and let $\ell \in \mathbb{N}$. Then there exists a complete acyclic coloring of $D$ using exactly $\ell$ colors if and only if $\vec{\chi}(D) \leq \ell \leq \adi(D)$.
\item Let $G$ be a graph and let $\ell \in \mathbb{N}$. Then there exists a complete arboreal coloring of $G$ using exactly $\ell$ colors if and only if $\va(G) \leq \ell \leq \ava(G)$.
\end{enumerate}
\end{theorem}
To prove the theorem, we consider generalizations of the dichromatic number respectively the vertex arboricity, where we want to minimize the number of colors in an acyclic coloring with the additional restriction that certain vertices must be colored the same.
\begin{definition} \label{partitions}
\noindent
\begin{enumerate}
\item Let $D$ be a digraph, and let $\mathcal{P}=\{P_1,\ldots,P_t\}$ be a partition of the vertex set such that $D[P_i]$ is acyclic for all $i \in [t]$. Then we define the \emph{$\mathcal{P}$-dichromatic number} of $D$, denoted by $\vec{\chi}_{\mathcal{P}}(D)$, to be the least number of colors required in a proper digraph coloring of $D$ (without monochromatic directed cycles) such that for any $i$, the vertices in $P_i$ receive the same color.
\item Let $G$ be a graph, and let $\mathcal{P}=\{P_1,\ldots,P_t\}$ be a partition of the vertex set such that $G[P_i]$ is a forest for all $i \in [t]$. Then we define the \emph{$\mathcal{P}$-vertex arboricity} of $G$, denoted by $\va_{\mathcal{P}}(G)$, to be the least number of colors required in a proper arboreal coloring of $G$ (without monochromatic cycles) such that for any $i$, the vertices in $P_i$ receive the same color.
\end{enumerate}
\end{definition}
We prepare the proof with the following simple lemma.
\begin{lemma} \label{change}
\noindent
\begin{enumerate}
\item Let $D$ be a digraph with a partition $\mathcal{P}=\{P_1,\ldots,P_t\}$ into acyclic vertex sets. Assume that also $D[P_1 \cup P_2]$ is acyclic, and let $\mathcal{Q}:=\{P_1 \cup P_2,P_3, \ldots, P_t\}$. Then we have
$$\vec{\chi}_\mathcal{P}(D) \leq \vec{\chi}_{\mathcal{Q}}(D) \leq \vec{\chi}_\mathcal{P}(D)+1.$$
\item Let $G$ be a graph with a partition $\mathcal{P}=\{P_1,\ldots,P_t\}$ into vertex sets inducing forests. Assume that also $G[P_1 \cup P_2]$ is a forest, and let $\mathcal{Q}:=\{P_1 \cup P_2,P_3, \ldots, P_t\}$. Then we have
$$\va_\mathcal{P}(G) \leq \va_{\mathcal{Q}}(G) \leq \va_\mathcal{P}(G)+1.$$
\end{enumerate}
\end{lemma}
\begin{proof}
We give a proof for the directed case, the undirected case is analogous. Clearly, any digraph coloring of $D$ which is compatible with $\mathcal{Q}$ is also compatible with $\mathcal{P}$. This directly yields $\vec{\chi}_\mathcal{P}(D) \leq \vec{\chi}_{\mathcal{Q}}(D)$. 

On the other hand, let $c:V(D) \rightarrow [\ell]$ be a digraph coloring of $D$ using $\ell=\vec{\chi}_\mathcal{P}(D)$ colors such that the vertices in $P_i$ are colored the same, for all $i \in [t]$. It is now easily seen that coloring all vertices in $P_3 \cup P_4 \cup \dots \cup P_t$ as in $c$ and giving color $\ell+1$ to all vertices in $P_1 \cup P_2$ defines a proper digraph coloring of $D$ which is compatible with $\mathcal{Q}$ and uses at most $\ell+1$ colors. This proves the inequality $\vec{\chi}_{\mathcal{Q}}(D) \leq \vec{\chi}_\mathcal{P}(D)+1$.
\end{proof}
\begin{proof}[Proof of Theorem \ref{interpolation}.]
We prove the first part of the Theorem, the proof of the second is completely analogous. 

First of all, the conditions on $\ell$ are necessary: Any complete acyclic coloring also is a proper digraph coloring, so it uses at least $\vec{\chi}(D)$ colors, and by definition, it cannot use more than $\adi(D)$. 

So let now $\ell \in \{\vec{\chi}(D),\ldots,\adi(D)\}$ be given. Define $\mathcal{P}_0:=\{\{v\}|v \in V(D)\}$ to be the partition of $V(D)$ into singletons, and let $\mathcal{P}$ denote the partition of $V(D)$ into the $\adi(D)$ color classes corresponding to a complete acyclic coloring of $D$ with the maximum number of colors. Looking at Definition \ref{partitions}, it is readily verified that $\vec{\chi}_{\mathcal{P}_0}(D)=\vec{\chi}(D)$ and $\vec{\chi}_{\mathcal{P}}(D)=\adi(D)$ (for the latter note that different partition classes must be colored differently, as their union is cyclic). Now consider a sequence $\mathcal{P}_0,\mathcal{P}_1,\mathcal{P}_2,\ldots,\mathcal{P}_r=\mathcal{P}$ consisting of partitions of $V(D)$ into acyclic vertex sets such that for any $i=0,1,\ldots,r-1$, $\mathcal{P}_{i+1}$ is obtained from $\mathcal{P}_i$ by merging a pair of partition classes. The existence is easily seen by successively splitting partition classes, starting from the back with $\mathcal{P}$. Applying Lemma \ref{change} we get that $\vec{\chi}_{\mathcal{P}_i}(D) \leq \vec{\chi}_{\mathcal{P}_{i+1}}(D) \leq \vec{\chi}_{\mathcal{P}_i}(D)+1$ for all $i \in [r-1]$. It follows from this and from $\vec{\chi}_{\mathcal{P}_0}(D) \leq \ell \leq \vec{\chi}_{\mathcal{P}_r}(D)$ that there exists some $i \in [r]$ such that $\ell=\vec{\chi}_{\mathcal{P}_i}(D)$. Let now $c:V(D) \rightarrow [\ell]$ denote an optimal digraph coloring compatible with $\mathcal{P}_i$ which uses the fewest number ($\ell$) of colors. This coloring must be a complete acyclic coloring: If the union of two color classes was acyclic, we could improve the number of colors used by merging these color classes, still keeping it compatible with $\mathcal{P}_i$. As we have found a complete acyclic coloring using exactly $\ell$ colors, this verifies the claim.
\end{proof}

\subsection{Behavior with respect to graph operations}\label{subsec:operations}

Most standard coloring parameters such as the chromatic number are monotone under subgraphs. This is not the case for the adichromatic number and a-vertex arboricity in general (consider a bidirected $C_4$ in the directed case and the multi-graph obtained from $C_4$ by doubling all edges for small examples, or alternatively the digraphs described in Proposition \ref{nobound} in the next section), we can establish a monotonicity under induced subgraphs.
\begin{lemma}
Let $D$ be a digraph and $G$ a graph.
\begin{itemize}
\item For any $v \in V(D)$, we have $\adi(D)-1 \leq \adi(D-v) \leq \adi(D)$.
\item For any $v \in V(G)$, we have $\ava(G)-1 \leq \ava(G-v) \leq \ava(G)$.
\end{itemize}
\end{lemma}
\begin{proof}
We prove the claim for the directed case, the undirected case is analogous. To prove that $\adi(D-v) \leq \adi(D)$, consider an optimal complete acyclic coloring of $D-v$ with color classes $V_1,\ldots,V_l$, $\ell:=\adi(D-v)$. If there is a color $i \in [\ell]$ such that $V_i \cup \{v\}$ induces an acyclic subdigraph, we can join $v$ to this color class in order to obtain a complete acyclic coloring of $D$ with $\ell$ colors. If on the other hand $V_i \cup \{v\}$ contains a directed cycle for every $i \in [\ell]$, we can give $v$ the new color $\ell+1$ and see that this defines a complete acyclic coloring of $D$ using $\ell+1$ colors. This implies that $\adi(D) \ge \ell=\adi(D-v)$ in every case. 

For the second inequality, we have to prove that $\adi(D)-1 \leq \adi(D-v)$. To see this, consider a complete acyclic coloring of $D$ using $r:=\adi(D)$ colors. There are at least $r-1$ color classes in this coloring which were not affected by the deletion of $v$ from $D$, and therefore, the union of any two of these color classes still contains a directed cycle. Clearly, each of the $r-1$ color classes still induces an acyclic subdigraph. Now simply give a unique new color to each vertex in $V(D-v)$ which is in none of the $r-1$ color classes. Perform a greedy merging process in which, as long as there exists a pair of color classes whose union is acyclic, we merge them. In the end, we obtain a complete coloring of $D-v$, in which no two of the $r-1$ color classes considered above was merged. Therefore, we have found a complete acyclic coloring of $D$ using at least $r-1=\adi(D)-1$ colors. This concludes the proof. 
\end{proof}
\begin{corollary} \label{InducedMonotonicity}
Let $D_1, D_2$ be digraphs and $G_1, G_2$ graphs.
\begin{itemize}
\item If $D_1$ is an induced subdigraph of $D_2$, then $\adi(D_1) \leq \adi(D_2)$.
\item If $G_1$ is and induced subgraph of $G_2$, then $\ava(G_1) \leq \ava(G_2)$.
\end{itemize}
\end{corollary}

A special role in the theory of acyclic digraph colorings is played by \emph{directed separations}. As a simplest example, a \emph{directed cut} in a digraph $D$ is a vertex set of the form $$\delta^+(X)=\{e \in E(D)|\text{tail}(e) \in X, \text{head}(e) \notin X\},$$ where $(X,\overline{X})$ is a partition of the vertex set into two non-empty parts such that no edge starts in $\overline{X}:=V(D) \setminus X$ and ends in $X$. If $S:=\delta^+(X)$ forms a directed cut, no directed cycle in $D$ can use an edge from $S$ and therefore either stays in $D[X]$ or in $D[\overline{X}]$, which implies that $\vec{\chi}(D)=\max\{\vec{\chi}(D[X]),\vec{\chi}(D[\overline{X}])\}$. Iterating this argument, one can see that the dichromatic number of a digraph can be computed as the maximum over the dichromatic numbers of its strongly connected components. For the adichromatic number, such a simple relation does not hold true, as even for the disjoint union of two digraphs, there is no explicit way of computing the adichromatic number in terms of the adichromatic numbers of the two components. However, we can bring it down to exactly this case (see also Proposition \ref{examples}, 4.).

\begin{observation} \label{dirsep}
Let $D$ be a digraph, and let $S \subseteq E(D)$ be a directed cut. Then $\adi(D)=\adi(D-S)$. 
\end{observation}
\begin{proof}
This follows directly from the fact that the complete acyclic colorings of $D$ are the same as those of $D-S$, because a vertex subset $X \subseteq V(D)$ is acyclic in $D$ if and only if it is acyclic in $D-S$.
\end{proof}

We can go one step further and consider cuts which are almost directed, i.e., they have only a single edge in forward-direction. In this case, we can contract this forward-edge without increasing the adichromatic number.
\begin{lemma}\label{almostdirected}
Let $D$ be a digraph with a non-trivial partition $(X, \overline{X})$ of the vertex set such that $\delta^+(X)=\{e\}$ for some $e \in E(D)$. Then we have $\adi(D/e) \le \adi(D)$, where $D/e$ is obtained from $D$ by identifying the endpoints of $e$.
\end{lemma}
\begin{proof}
Let $c:V(D/e) \rightarrow [\ell]$ be a complete acyclic coloring of $D$ using $\ell=\adi(D/e)$ colors. Let $c'$ be the vertex-coloring of $D$ in which every vertex not incident to $e$ is colored as by $c$, and where the endpoints of $e$ both receive the color which is given to the contraction vertex of $e$ under $c$. It is clear that this still defines an acyclic coloring, as any directed cycle in $D$, after contracting $e$, still yields a directed cycle in $D/e$ using exactly the same colors. On the other hand, if $i \neq j \in [\ell]$ is a pair of colors, then there exists a directed cycle in $D/e$ which uses only colors $i$ and $j$. Re-inserting the arc $e$ in case the cycle uses the contraction vertex of $e$ now defines a directed cycle in $D$ which also only uses colors $i$ and $j$. We therefore have found a complete acyclic coloring of $D$ which uses $\ell$ colors. This proves the claim. 
\end{proof}
This operation generalizes the so-called \emph{butterfly-contractions}, which have been investigated in structural digraph theory. An edge $e \in E(D)$ is called \emph{butterfly-contractible}, if it is the unique emanating edge of its tail or the unique edge entering its head. A \emph{butterfly minor} of a digraph $D$ is obtained by repeated contractions of butterfly-contractible edges from a subdigraph of $D$. The following is now a direct consequence of Corollary \ref{InducedMonotonicity} and Lemma \ref{almostdirected}:
\begin{corollary}
Let $D_1$ be an induced butterfly-minor of $D_2$, i.e., $D_1$ is obtained from an induced subdigraph of $D_2$ by repeatedly contracting butterfly-contractible edges. Then 
$$\adi(D_1) \leq \adi(D_2).$$
\end{corollary}

The following statements, which yield lower bounds on the a-vertex arboricity of a graph in terms of the a-vertex arboricities of induced minors, will form a central tool in the proof of our main result, Theorem \ref{FVlowerbound}.

\begin{lemma} \label{treecontraction}
Let $G$ be a (multi-)graph and let $T=G[X]$ be an induced subtree of $G$. Let $G/T$ denote the (multi-)graph obtained from $G$ by deleting all edges of $T$ from $G$ and identifying $X$ into a single vertex $v_X$. Then 
$$\ava(G/T) \leq \ava(G).$$
\end{lemma}
\begin{proof}
Let $\ell:=\ava(G/T)$ and let $c:V(G/T) \rightarrow [\ell]$ be a complete arboreal coloring of $G/T$ using all $\ell$ colors. We claim that the coloring $c':V(G) \rightarrow [\ell]$, 
$$c'(v):=\begin{cases} c(v), & \text{if }v \not\in X \cr
c(v_X), & \text{if }v \in X  \end{cases},$$ is also a complete arboreal coloring of $G$ that uses $\ell$ colors. For this purpose, we must verify that there are no monochromatic cycles in $G$ with respect to $c$ and that in the union of any two color classes, there is a cycle. For the first part, suppose there was a cycle $C$ in $G$ all whose vertices are colored $i$ in $c'$. Because $T$ is an induced tree, $C$ must contain a vertex outside $X$. Consequently, after the contraction of $X$, the cycle $C$ yields a monochromatic closed walk of positive length in $G/T$, which by definition of $c'$ must still be monochromatic, a contradiction. On the other hand, given any pair $i \neq j \in [\ell]$ of colors, there is a cycle in $G/T$ which uses only colors $i$ and $j$ according to $c$. If the cycle does not use $v_X$, this yields also a cycle in $G$ which only uses colors $i$ and $j$ according to $c'$, as desired. In the case that the cycle traverses $v_X$, let $e,f$ be the two incident edges of $v_X$ on the cycle. By connecting the endpoints of $e$ and $f$ in $T$ by the unique monochromatic connection path in $T$ if necessary, we find that also in this case there is a cycle in $G$ which uses only colors $i,j$ according to $c'$. Hence, $c'$ defines a complete arboreal coloring of $G$ with $\ell$ colors, and $\ava(G) \ge \ell=\ava(G/T)$.
\end{proof}
\begin{corollary} \label{inducedsubdivisions}
Let $G$ and $H$ be simple graphs such that $G$ contains an induced subdivision of $H$. Then $\ava(G) \ge \ava(H)$.
\end{corollary}
\begin{proof}
Repeated application of Lemma \ref{treecontraction} to contractions of subdivision edges yields that the a-vertex arboricity of any subdivision of a graph is lower bounded by the a-vertex arboricity of the graph itself. The claim now follows from Corollary \ref{InducedMonotonicity}.
\end{proof}
\section{Upper Bounds for Minimum Feedback Vertex Sets} \label{main}
The main goal of this section is to complement the lower bounds on $\tau$ via $\ava$ and $\adi$ from Proposition~\ref{FVupperbound} with upper bounds. 

By Observation \ref{achrom}, the adichromatic number of the directed biorientation $\overset{\leftrightarrow}{K}_{n,n}$ of the complete bipartite graph is given by the achromatic number of $K_{n,n}$, which is $2$. However, the size of a smallest feedback vertex set equals $n$. Similarly, the undirected biorientation $2K_{n,n}$ has a-vertex arboricity $2$ but $\tau(2K_{n,n})=n$ for any $n \ge 1$. In the rest of this section, we therefore focus on simple digraphs and graphs and demonstrate that while there are simple digraphs $D$ with bounded adichromatic number and unbounded $\tau(D)$ (Proposition \ref{nobound}), $\tau(G)$ is upper bounded in terms of $\ava(G)$ for simple graphs (Theorem \ref{FVlowerbound}). On the way to this result we achieve an upper bound for the degeneracy by a function of $\ava$ (Corollary~\ref{degeneracy}). Moreover, we show that in minor-closed classes $\tau(D)\leq f(\adi(D))$ for some function $f$ (Theorem~\ref{minor-closed}).
Corollaries of the above include non-trivial relations between $\ava$ and $\adi$ in Subsection~\ref{subsec:relations}.

The following construction gives a family of simple digraphs with an unbounded size of the feedback vertex set but bounded adichromatic number. Additionally, these digraphs can have arbitrarily large directed girth. 

Let $D(n,k)$ with $n \ge 1, k \ge 3$ denote a cyclically oriented Tur\'{a}n graph, that is, the $k$-partite digraph whose vertex set consists of $k$ disjoint partition classes $V_1,\ldots,V_k$ of size $n$ each and where $E(D)=\bigcup_{i=1}^{k}{(V_i \times V_{i+1})}$ ($k+1:=1$). Hence $D(n,k)$ is obtained from $\vec{C}_k$ by replacing each vertex by $n$ independent copies.

\begin{proposition}\label{nobound}
For any $n \ge 1, k \ge 3$, we have $\adi(D(n,k)) \leq k$ while $\tau(D(n,k))=n$.
\end{proposition}
\begin{proof}
Clearly, we can find a packing of $n$ vertex-disjoint directed cycles in $D(n,k)$, and so $\tau(D(n,k)) \ge n$. On the other hand, $V_1$ forms a feedback vertex set, and we conclude that $\tau(D(n,k))=n$. To see that $\adi(D(n,k)) \le k$, let $c:V(D(n,k)) \rightarrow [\ell]$ be a complete acyclic coloring using $\ell$ colors, and assume towards a contradiction that $\ell \geq k+1$. For each $i \in [\ell]$, there is at least one partition class in which $i$ does not appear. By the pigeon-hole principle, we therefore find a pair $c_1 \neq c_2 \in [\ell]$ of colors such that both do not appear in a certain partition class. However, there must be a directed cycle in $D(n,k)$ using only vertices with color $i$ or $j$. This contradiction shows $\adi(D(n,k)) \leq k$. 
\end{proof}
The rest of this paragraph is devoted to prove an inverse relation between the a-vertex arboricity and the smallest size of a feedback vertex set for undirected simple graphs.
\begin{theorem}\label{FVlowerbound}
There is a function $f:\mathbb{N} \rightarrow \mathbb{N}$ such that for any simple graph $G$, we have $$\tau(G) \leq f(\ava(G)).$$
\end{theorem}
We prepare the proof with some helpful statements. For our next result we will combine Corollary~\ref{inducedsubdivisions} with the following  strong result from \cite{kuehn}:\begin{theorem}[\cite{kuehn}]\label{Hammer}
For any $s \in \mathbb{N}$ and any simple graph $K$ there is some $d=d(s,K) \in \mathbb{N}$ such that every simple graph $G$ with minimum degree greater than $d$ contains $K_{s,s}$ as a subgraph or an induced subdivision of $K$.
\end{theorem}

For a simple graph $G$, the \emph{degeneracy} of $G$ is defined as $\deg(G):=\max_{H \subseteq G}{\delta(G)}$, where the maximum is taken over all subgraphs (or, equivalently, all induced subgraphs) of $G$. 
We show that simple graphs of bounded a-vertex-arboricity have bounded degeneracy.

\begin{corollary} \label{degeneracy}
There exists a function $g:\mathbb{N} \rightarrow \mathbb{N}$ such that $\deg(G) \leq g(\ava(G))$ for every simple graph $G$.
\end{corollary}
\begin{proof}
For any $k \in \mathbb{N}$, define $g(k)$ as the integer $d(s,K)$ from Theorem \ref{Hammer} where $s=k+1$ and $K=K_{k+1,k+1}$.

Now let $G$ be an arbitrary simple graph and let $k:=\ava(G)$. We have to prove that $\deg(G) \leq g(k)$. Assume towards a contradiction that $\deg(G)>g(k)$, i.e., there exists an induced subgraph $H$ of $G$ such that the minimum degree of $H$ is greater than $g(k)=d(k+1,K_{k+1,k+1})$. By Corollary \ref{InducedMonotonicity}, we have $k=\ava(G) \ge \ava(H)$. By Theorem \ref{Hammer}, $H$ either contains $K_{k+1,k+1}$ as a subgraph or as an induced subdivision.
In the first case, let $X \subseteq V(H)$ be the set of vertices of the subgraph. On the one hand, we know that $k \ge \ava(H) \ge \ava(H[X])$. On the other hand, $H[X]$ is a simple graph which contains $K_{k+1,k+1}$ as a spanning subgraph. Consider a perfect matching of $K_{k+1,k+1}$ and color the vertices in $X$ with $k+1$ different colors, such that end vertices of the same matching edge have the same color, and all matching edges are colored differently. It is now easily seen that this defines a complete arboreal coloring of $H[X]$ with more than $k$ colors, which yields the desired contradiction in this case. In the second case we directly apply Corollary \ref{inducedsubdivisions} to obtain the contradiction $k \ge \ava(H) \ge \ava(K_{k+1,k+1})=k+1$.
\end{proof}
The last ingredient of our proof is the following well-known theorem of Erd\H{o}s and P\'{o}sa, which relates the maximum size of a vertex-disjoint cycle packing in a graph to the minimum size of a feedback vertex set. For a graph $G$ let $\nu(G)$ denote the maximal size of a collection of pairwise vertex-disjoint cycles in $G$.
\begin{theorem}[\cite{erdosposa}]\label{epos}
There is an absolute constant $c>0$ such that for every $k \in \mathbb{N}$, every graph $G$ with $\tau(G)>ck\log(k)$ fulfills $\nu(G) \ge k$.
\end{theorem}

We are now prepared for the proof of Theorem \ref{FVlowerbound}.
\begin{proof}[Proof of Theorem \ref{FVlowerbound}]
For a clearer presentation, we prove the theorem by contradiction. From a finer analysis of the proof, one could derive an explicit expression for $f(k)$, however, the bound would be rather bad. So assume for the rest of the proof that such a function $f$ as claimed does not exist. This means that there is a fixed $A \in \mathbb{N}$ and an infinite sequence $(G_s)_{s=1}^{\infty}$ of simple graphs such that $\ava(G_s)<A$ for all $s \in \mathbb{N}$ but $\tau(G_s) \rightarrow \infty$. From Theorem \ref{epos} we directly conclude that also $\nu(G_s) \rightarrow \infty$.

From Corollary \ref{degeneracy} we get that there exists a constant $d:=\max_{l=1,\ldots,A-1}{g(l)}>0$ such that all the graphs $G_s$ are $d$-degenerate. 

For each $s \ge 1$, we fix a packing $\mathcal{C}_s$ of induced (that is, chordless) and pairwise vertex-disjoint cycles in $G_s$ of size $\nu(G_s)$. 

For each $s \ge 1$, we associate with $\mathcal{C}_s$ a model graph $M_s$ which has $|\mathcal{C}_s|$ vertices, one for each cycle in $\mathcal{C}_s$, and an edge between two vertices for every edge spanned between the corresponding cycles in $G_s$ (so this might be a multi-graph). Because the cycles were assumed to be induced, we know that $M_s$ is loopless.
\paragraph{Claim: 
$\alpha(M_s) < \binom{A}{2}$ for all $ s\ge 1$.}
\begin{proof}
Assume towards a contradiction the statement was false. Consequently, we can find some $s \ge 1$ such that $M_s$ contains an independent set $I$ of size $\binom{A}{2}$. Let $H$ be the subgraph of $G_s$ induced by the union of the vertex sets of cycles in $\mathcal{C}_s$ corresponding to the vertices in $I$. Consider some bijection which maps each pair $\{i,j\} \in \binom{[A]}{2}$ to one of the $\binom{A}{2}$ cycles corresponding to $I$. 

We now define a vertex-coloring of $H$ as follows: For each cycle associated with the pair $\{i,j\}$, we partition its vertex set into two non-trivial subsets (say induced paths), one of which gets color $i$, while the other gets color $j$. We claim that this defines a complete arboreal coloring of $H$: First of all, every color class induces a forest on each of the cycles it appears on, and since $I$ was independent, there are no edges between the considered cycles which could create monochromatic cycles. Moreover, for each pair $i,j \in [A]$ of colors, the union of the corresponding color classes contains a cycle, namely the one with label $\{i,j\}$. This proves $A>\ava(G_s) \ge \ava(H) \ge A$, which is the desired contradiction.
\end{proof}
Applying Ramsey's Theorem to each of the graphs $M_s,  s \ge 1$, we find that $$R\left(\omega(M_s)+1,\binom{A}{2}\right)>|V(M_s)|=\nu(G_s) \rightarrow \infty,$$
where for any $r,b \in \mathbb{N}, r, b \ge 1$, $R(r,b)$ denotes the well-known Ramsey number. Therefore, the size $\omega(M_s)$ of a maximum clique in $M_s$ tends to infinity for $s \rightarrow \infty$. For each $s \ge 1$, consider a clique $W_s$ in $M_s$ of maximum size and let $G_s'$ be the subgraph of $G_s$ induced by the vertices contained in the cycles corresponding to the vertices $W_s$ of $M_s$. Clearly, the sub-collection $\mathcal{C}_s'$ of $\mathcal{C}_s$ corresponding to $W_s$ defines a decomposition of $G_s'$ in induced vertex-disjoint cycles of size $|\mathcal{C}_s'|=\omega(W_s) \rightarrow \infty$ in $G_s'$. Moreover, we have $\ava(G_s') \leq \ava(G_s)<A$ and $\deg(G_s') \leq \deg(G_s) \leq d$ for all $s \in \mathbb{N}$. In the following, we will continue working with the sequence $(G_s')_{s=1}^{\infty}$ of simple graphs. 

For a fixed $s \ge 1$ consider the graph $G_s'$ with the cycle-decomposition $\mathcal{C}_s'=\{C_1,\ldots,C_k\}$. By the definition of $\mathcal{C}_s'$, for every pair $C_j,C_l$ of cycles, there is an edge $e_{jl} \in E(G_s')$ with endpoints in $V(C_j)$ and $V(C_l)$. 
\paragraph{Claim: There are less than $R:=R(2A,A)$ cycles $C \in \mathcal{C}_s'$ with $|V(C)| \ge R$.
}
\begin{proof}
Assume towards a contradiction that there were at least $R$ cycles in $\mathcal{C}_s'$ with at least $R$ vertices each, say $C_1,\ldots,C_{R}$. For each $i \in [R]$, we can find a vertex $v_i \in V(C_i)$ which is not incident to any of the edges $\{e_{jl}| j, l \in [R]\}$. Let $X:=\bigcup_{i=1}^{R}{(V(C_i) \setminus \{v_i\})}$ and consider the induced subgraph $G_s'[X]$. For every $i$, let $P_i:=C_i-v_i$. $P_1,\ldots,P_R$ defines a vertex-partition of $G_s'[X]$ into induced paths. Let $M_X$ be the model (multi-)graph on $R$ vertices obtained from $G_s'[X]$ by identifying each of $P_1, \ldots, P_{R}$ into a single vertex. By Corollary \ref{InducedMonotonicity} and Lemma \ref{treecontraction}, we know that $\ava(M_X) \leq \ava(G_s'[X]) \leq \ava(G_s')<A$. Because all the edges $e_{jl}, 1 \leq j<l \leq R$ still exist in $G_s'[X]$, we know that the vertices of $M_X$ are mutually adjacent. Now define a $2$-coloring of the pairs of vertices of $M_X$ where a pair is colored blue if there is a simple edge between the corresponding vertices and red if there are at least two parallel edges between the corresponding vertices. By Ramsey's theorem, we can find $2A$ vertices in $M_X$ spanning a blue clique or $A$ vertices spanning a red clique. In the first case, we directly have the contradiction $A>\ava(M_X) \ge \ava(K_{2A})=A$, as desired, while in the second case, we find an induced subgraph of $M_X$ on $A$ vertices in which each pair of vertices is connected by a bigon. This subgraph has a-vertex arboricity $A$ as well (color each vertex with a different color), which yields the contradiction also in this case.
\end{proof}
For each $s \ge 1$, consider the subset $\mathcal{C}_s^R \subseteq \mathcal{C}_s'$ of cycles of length less than $R$, and consider the induced subgraph $H_s$ of $G_s'$ with vertex set $\bigcup_{C \in \mathcal{C}_s^R}{V(C)}$. For each $s \ge 1$, define $k_s:=|\mathcal{C}_s^R| > |\mathcal{C}_s'|-R$. By the above, we have $k_s \rightarrow \infty$ for $s \rightarrow \infty$. Because $G_s'$ is $d$-degenerate, so is $H_s$, and therefore we have $|E(H_s)| \leq dn_s$, where $n_s$ is the number of vertices of $H_s$. By definition, we have $n_s=\sum_{C \in \mathcal{C}_s^R}{|V(C)|} \leq Rk_s$. On the other hand, all the distinct edges $e_{jl}$ with $C_j, C_l \in \mathcal{C}_s^R$ are contained in $E(H_s)$, and so we get the estimate
$$\binom{k_s}{2} \leq |E(H_s)| \leq dRk_s$$ for all $s \ge 1$. This clearly contradicts the fact that $k_s$ can grow arbitrarily large. This concludes the proof of the theorem.
\end{proof}
The given examples for digraphs and multi-graphs with small complete coloring parameters but without small feedback vertex sets are based on very dense (di)graphs. However, for many investigations, \emph{minor-closed} classes of graphs such as planar graphs, which are rather sparse, are also important. In the following we show that for orientations of graphs in a fixed non-trivial minor-closed class, also for digraphs it is possible to establish an upper bound on the feedback vertex set in terms of the adichromatic number. Moreover, we give more explicit bounds for undirected graphs within such a class. 

To prove the first part of the next Theorem, we need a directed version of Theorem \ref{epos}. This result is not trivial at all. Before its resolution in \cite{Reed1996}, it was known as Younger's Conjecture. No good (polynomial) upper bounds on the function $g$ are known yet. Again, $\nu(D)$ denotes the maximal size of a collection of pairwise vertex-disjoint directed cycles in $D$. 
\begin{theorem}[\cite{Reed1996}] \label{younger}
There exists a function $g:\mathbb{N} \rightarrow \mathbb{N}$ such that for any digraph $D$, we have 
$$\tau(D) \leq g(\nu(D)).$$
\end{theorem}
\begin{theorem} \label{minor-closed}
Let $\mathcal{G}$ be a minor-closed class of simple graphs which is non-trivial (that is, it does not contain all graphs).
\begin{enumerate}
\item There is a function $f: \mathbb{N} \rightarrow \mathbb{N}$ (depending on $\mathcal{G}$), such that for any digraph $D$ whose simple underlying graph (obtained from ignoring parallel edges) is contained in $\mathcal{G}$, we have
$$\tau(D) \leq f(\adi(D)).$$
\item There is a constant $C>0$ (depending on $\mathcal{G}$) such that for every graph $G$ whose simplification (identifying parallel edges) lies in $\mathcal{G}$, we have
$$\tau(G) \leq C \cdot \ava(G)^2 \log(\ava(G))$$
\end{enumerate} 
\end{theorem}
\begin{proof}
We start by noting that the graphs in $\mathcal{G}$ have bounded chromatic number: Since $\mathcal{G}$ is non-trivial, there is a graph $H$ which is a forbidden minor for all members of of $\mathcal{G}$. Therefore, all graphs in $\mathcal{G}$ are $K_{|V(H)|}$-minor free. By a classical result of Mader (\cite{mader67}), these graphs have bounded degeneracy and therefore bounded chromatic number. Let in the following $d>0$ denote a constant such that $\chi(G) \leq d$ for all $G \in \mathcal{G}$. It follows from the estimate $\frac{|V(G)|}{\alpha(G)} \leq \chi(G)$ that $\alpha(G) \ge \frac{1}{d}n$ for all $G \in \mathcal{G}$ on $n$ vertices.
\begin{enumerate}
\item Let $g:\mathbb{N} \rightarrow \mathbb{N}$ be the function from Theorem \ref{younger}, and let $D$ be a given digraph whose simplified underlying graph is in $\mathcal{G}$. Let $\nu(D)=k$ and let $\{C_1,\ldots, C_k\}$ be an optimal collection of vertex-disjoint and w.l.o.g induced directed cycles. Consider the simple model-graph $M$ which has $k$ vertices, one for each cycle $C_i$, and an edge between two vertices if the corresponding cycles are connected by an edge. Because the cycles $C_i$ are all induced, $M$ is obtained from the simplified underlying graph of $D$ by first deleting all the vertices not on one of the cycles and then contracting the cycles into vertices. These are graph minor operations, and therefore we have $M \in \mathcal{G}$. We conclude that $\alpha(M) \ge \frac{1}{d}k$. Let $I \subseteq V(M)$ be an independent set of size $|I| \ge \frac{1}{d}k$ in $M$ and consider the subdigraph $D'$ of $D$ induced by the union of the vertex sets of cycles $C_i \in \mathcal{C}$ corresponding to vertices in $I$. We know that $D'$ is the disjoint union of at least $\frac{1}{d}k$ directed cycles. By Proposition \ref{examples} and Corollary \ref{InducedMonotonicity}, we have that $\frac{1}{d}k < \binom{\adi(D')+1}{2} \leq \binom{\adi(D)+1}{2}$. We finally conclude that (assuming $g$ to be monotone)
$$\tau(D) \leq g(\nu(D)) \leq g\left(d\binom{\adi(D)+1}{2}\right)=:f(\adi(D)),$$ which proves the claim.
\item The proof works completely analogous to the undirected case, and we obtain the estimate 
$$\nu(G)<d\binom{\ava(G)+1}{2} \leq d \cdot \ava(G)^2.$$ Finally, this implies using Theorem \ref{epos} that
$$\tau(G) \leq c \cdot \nu(G)\log(\nu(G)) \leq C \cdot \ava(G)^2 \log(\ava(G)),$$ where $C>0$ is a constant which only depends on $\mathcal{G}$. 
\end{enumerate}
\end{proof}

\subsection{Interplay of $\ava$ and $\adi$}\label{subsec:relations}

As a direct consequence of Theorem \ref{FVlowerbound} we can prove a one-sided relationship between the adichromatic number of a simple digraph and the a-vertex arboricity of its underlying graph.
\begin{corollary}\label{interplay}
There exists a function $h_1:\mathbb{N} \rightarrow \mathbb{N}$ such that for any simple digraph $D$ with underlying graph $G$, we have
$$\adi(D) \leq h_1(\ava(G)).$$
\end{corollary}
\begin{proof}
This follows directly from Proposition \ref{FVupperbound} (applied to $D$) and Theorem \ref{FVlowerbound}, as we have
$$\adi(D) \leq \tau(D)+1 \leq \tau(G)+1 \leq f(\ava(G))+1.$$
\end{proof}

The above estimate cannot be reversed when looking at a fixed digraph (consider acyclic digraphs). However, if a graph has large a-vertex arboricity, it is possible to find an orientation of $G$ with large adichromatic number.
\begin{proposition}\label{outerplay}
There is a function $h_2:\mathbb{N} \rightarrow \mathbb{N}$ such that 
$$\ava(G) \leq h_2\left(\max_{D \in \mathcal{O}(G)}\adi(D)\right)$$ for every graph $G$, where $\mathcal{O}(G)$ denotes the set of digraphs whose underlying graph is $G$.
\end{proposition}
\begin{proof}
Because $\ava(G)$ is bounded in terms of $\tau(G)$, which again (by Theorem \ref{epos}) is upper bounded by a function of $\nu(G)$, it suffices to prove that graphs with a sufficiently large cycle packing have an orientation with large adichromatic number. So let $k \in \mathbb{N}$ be arbitrary and let $G$ be a graph with $\nu(G) \ge \binom{k}{2}$. Consider a cycle packing $\mathcal{C}=\left\{C_1,\ldots,C_{\binom{k}{2}}\right\}$ of induced cycles and let $D$ be the orientation of $G$ in which all edges with endpoints in $V(C_i)$ and $V(C_j)$ are oriented towards $C_j$ if $i<j$, and where all cycles $C_i$ are made directed. Because each pair of cycles is separated by a directed edge-cut in $D$, we conclude from Observation \ref{dirsep} that $\adi(D)=\adi(D')$, where $D'$ is obtained from $D$ by deleting all edges between the cycles. As $D'$ is the disjoint union of $\binom{k}{2}$ directed cycles, we conclude from Proposition \ref{examples} that $\adi(D') \ge k$. As $k$ was arbitrary, this proves the claim.
\end{proof}

\section{Discussion}\label{sec:discuss}

In this final section we want to touch upon two topics. First, we take a more general point of view by relating our investigations to the Erd\H{o}s-P\'osa property and introducing the novel notion of $\tau$-boundedness. Second, we comment on the behavior of the parameters with respect to randomness. 

\subsection{Erd\H{o}s-P\'osa and $\tau$-boundedness}\label{subsec:erdosposa}
Many of the results in this paper fit into the following more general setting. Let $\mathcal{H}$ be a class of \emph{guest (di)graphs}, then for a \emph{host (di)graph} define $\tau_{\mathcal{H}}(G)$ to be the size of a minimum $F\subseteq V$ such that $G-F$ is $\mathcal{H}$-free, that is, it contains no element of $\mathcal{H}$ as an induced sub(di)graph. Moreover, denote by $\nu_{\mathcal{H}}(G)$ the size of a largest packing of elements of $\mathcal{H}$ as induced sub(di)graphs in $G$. While clearly $\nu_{\mathcal{H}}(G)\leq \tau_{\mathcal{H}}(G)$ for all (di)graphs $G$, one says that a host class $\mathcal{G}$ has the \emph{Erd\H{o}s-P\'osa property} with respect to $\mathcal{H}$ if there is a function $f$ such that $\tau_{\mathcal{H}}(G)\leq f(\nu_{\mathcal{H}}(G))$ for all $G\in \mathcal{G}$. Define an \emph{$\mathcal{H}$-coloring} of $G$ to be a partition of $V(G)$ into sets that induce $\mathcal{H}$-free (di)graphs. Call an $\mathcal{H}$-coloring \emph{complete} if the union of any two color classes contains a member of $\mathcal{H}$ as induced sub(di)graph. The \emph{$\mathcal{H}$-chromatic number} $\chi_{\mathcal{H}}(G)$ and the \emph{$\mathcal{H}$-achromatic number} $\Psi_{\mathcal{H}}(G)$ are the smallest respectively the largest number of colors that can be used in a complete $\mathcal{H}$-coloring of $G$. With a completely analogous proof to Theorem~\ref{interpolation} we get an interpolation theorem.

\begin{theorem}\label{geninterpolation}
Let $G$ be a (di)graph and let $\ell \in \mathbb{N}$. Then there exists a complete $\mathcal{H}$-coloring of $G$ using exactly $\ell$ colors if and only if $\chi_{\mathcal{H}}(G)\leq\ell\leq \Psi_{\mathcal{H}}(G)$.
\end{theorem}

More importantly, the arguments of Proposition~\ref{FVupperbound} go through to show:
\begin{proposition} \label{genFVupperbound}
For any (di)graph $G$, we have $\Psi_{\mathcal{H}}(G) \leq \tau_{\mathcal{H}}(G)+1$.
\end{proposition}

Conversely, we say that a host class $\mathcal{G}$ is \emph{$\tau_{\mathcal{H}}$-bounded} if there is a function $f$ such that $\tau_{\mathcal{H}}(G)\leq f(\Psi_{\mathcal{H}}(G))$ for all $G\in\mathcal{G}$. 
By pairing the parts of a maximum complete $\mathcal{H}$-coloring one obtains a set of $\lfloor \frac{\Psi_{\mathcal{H}}(G)}{2} \rfloor$ disjoint members of $\mathcal{H}$ in $G$. Thus, ${\Psi_{\mathcal{H}}(G)}\leq {2}\nu_{\mathcal{H}}(G)+1$ and we get that $\tau_{\mathcal{H}}$-boundedness is a strengthening of the Erd\H{o}s-P\'osa property:
\begin{proposition}\label{strengthening}
 If $\mathcal{G}$ is $\tau_{\mathcal{H}}$-bounded, then $\mathcal{G}$ has the Erd\H{o}s-P\'osa property.
\end{proposition}

A classical result for the achromatic number states that the size of a minimum vertex cover is bounded in terms of the achromatic number of a graph (\cite{farber}), i.e, the class $\mathcal{G}$ of all graphs is $\tau_{K_2}$-bounded.
Theorem~\ref{FVlowerbound} shows that the class of simple graphs is also $\tau_{\mathcal{C}}$-bounded with respect to the class $\mathcal{C}$ of cycles, thus in a sense strengthening the classical Erd\H{o}s-P\'osa result~\cite{erdosposa}. Furthermore, Theorem~\ref{minor-closed} can be generalized in a straight-forward way to yield:

\begin{theorem} \label{genminor-closed}
Let $\mathcal{G}$ be (the orientations of) a non-trivial minor-closed class of simple undirected graphs. If $\mathcal{G}$ has the Erd\H{o}s-P\'{o}sa property, and if the members of $\mathcal{H}$ are weakly connected, then $\mathcal{G}$ is $\tau_{\mathcal{H}}$-bounded.
\end{theorem}

On the other hand, while the class $\mathcal{G}$ of all digraphs  has the Erd\H{o}s-P\'osa property with respect to the class $\vec{\mathcal{C}}$ of directed cycles, see~\cite{Reed1996}, our construction in Proposition~\ref{nobound} shows that $\mathcal{G}_g$, the class of digraphs with directed girth at least $g$, is not $\tau_{\vec{\mathcal{C}}}$ -bounded, for any fixed $g \ge 2$. Hence, the strengthening of the Erd\H{o}s-P\'osa property claimed in Proposition~\ref{strengthening} is strict and leaves open finding good lower bounds for the adichromatic number. In general, we believe that $\tau$-boundedness deserves further investigation in particular with respect to the plenty of Erd\H{o}s-P\'osa properties that have been studied, see e.g.~\cite{EPweb}.

\subsection{A-vertex-arboricity of a random graph}\label{subsec:random}
Let $G(n,p)$ for $n \in \mathbb{N}, p \in (0,1)$ denote a graph generated randomly according to the Erd\H{o}s-R\'{e}nyi model by taking $n$ vertices, and connecting a pair of vertices with probability $p$, independently from all other pairs of vertices. The following bounds follow directly from a result of Boll\'{o}bas on the chromatic number of a random graph (\cite{Bollobas1988}). They show that for a fixed probability $p$ (independent of $n$), the a-vertex arboricity and the minimum size of a feedback vertex set are at most a logarithmic factor apart for asymptotically almost all graphs.
\begin{proposition}
There are constants $c_1,c_2>0$ such that for any fixed $p \in (0,1)$, we have that a.a.s.
$$\ava(G(n,p)) \ge c_1\log\left(\frac{1}{1-p}\right)\frac{n}{\log n},$$ $$\tau(G(n,p)) \leq n-c_2\frac{\log(n)}{\log(\frac{1}{1-p})}.$$
\end{proposition}
\begin{proof}
The main result from \cite{Bollobas1988} shows that a.a.s., we have
$$\chi(G(n,p))=\left(\frac{1}{2}+o(1)\right)\log\left(\frac{1}{1-p}\right)\frac{n}{\log n}.$$ The claim now follows from the simple estimates
$$\ava(G) \ge \va(G) \ge \frac{1}{2}\chi(G)$$ and
$$\tau(G)=n-\max\{|V(F)|:F \text{ is induced forest}\} \leq n-\alpha(G) \leq n-\frac{n}{\chi(G)}$$ 
which hold for any simple graph $G$.
\end{proof}
This result gives a hint that in fact, it might be possible to find good upper bounds for the function $f:\mathbb{N} \rightarrow \mathbb{N}$ from Theorem \ref{FVlowerbound}, which are much better than the multiply exponential bounds one would obtain from our proof. The best lower bound we know so far is from the simple example in Proposition \ref{examples}, (4), which only shows that $f(k)=\Omega(k^2)$.  We would be very interested in any improvements of bounds (lower and upper) on the asymptotic growth of the function $f$. Another question goes towards the bevaviour of the adichromatic number on random digraphs.

\paragraph*{Acknowledgments:} The authors wish to thank Mika Olsen and Gabriela Araujo-Pardo for discussions after their talk about diachromatic numbers at ACCOTA 2018. This research was initiated during a stay of Raphael Steiner and Kolja Knauer at FernUniversit\"at Hagen and continued at Technische Universit\"at Berlin. Kolja Knauer was partially supported by both of the former and ANR grant DISTANCIA: ANR-17-CE40-0015.

\bibliography{references}

\begin{thebibliography}{APMBORM18}

\bibitem[APMBORM18]{diachromatic}
Gabriela Araujo-Pardo, Juan~José Montellano-Ballesteros, Mika Olsen, and
  Christian Rubio-Montiel.
\newblock The diachromatic number of digraphs.
\newblock {\em The Electronic Journal of Combinatorics}, 25 (3), 2018.

\bibitem[{Bol}88]{Bollobas1988}
Bela {Bollob{\'a}s}.
\newblock The chromatic number of random graphs.
\newblock {\em Combinatorica}, 8:49--55, 1988.

\bibitem[CK69]{10.1112/jlms/s1-44.1.612}
Gary Chartrand and Hudson~V. Kronk.
\newblock {The Point-Arboricity of Planar Graphs}.
\newblock {\em Journal of the London Mathematical Society}, s1-44(1):612--616,
  1969.

\bibitem[CKW68]{Chartrand1968}
Gary Chartrand, Hudson~V. Kronk, and Curtiss~E. Wall.
\newblock The point-arboricity of a graph.
\newblock {\em Israel Journal of Mathematics}, 6(2):169--175, 1968.

\bibitem[{Edw}97]{survey2}
Keith~J. {Edwards}.
\newblock {The harmonious chromatic number and the achromatic number}.
\newblock {\em Cambridge University Press}, pages 13--47, 1997.
\newblock {In: Surveys in Combinatorics 1997 (Invited papers for the 16th
  British Combinatorial Conference)}.

\bibitem[{Edw}13]{harmonious}
Keith~J. {Edwards}.
\newblock {Harmonious chromatic number of directed graphs.}
\newblock {\em {Discrete Appl. Math.}}, 161(3):369--376, 2013.

\bibitem[EP65]{erdosposa}
Paul Erd\H{o}s and Lajos P\'{o}sa.
\newblock On independent circuits contained in a graph.
\newblock {\em Canadian Journal of Mathematics}, 17:347--352, 1965.

\bibitem[EPw]{EPweb}
Dynamic {E}rd{\H{o}}s-{P}\'osa listing.
\newblock \url{http://www.user.tu-berlin.de/jraymond/Erd%C5%91s-P%C3%B3sa/}.
\newblock Maintained by Jean-Florent Raymond.

\bibitem[FHHM86]{farber}
Martin Farber, Gen\v{a} Hahn, Pavol Hell, and Donald Miller.
\newblock Concerning the achromatic number of graphs.
\newblock {\em Journal of Combinatorial Theory, Series B}, 40:21 -- 39, 1986.

\bibitem[HHP67]{achromatic}
Frank {Harary}, Stephen {Hedetniemi}, and Geert {Prins}.
\newblock {An interpolation theorem for graphical homomorphisms.}
\newblock {\em {Port. Math.}}, 26:453--462, 1967.

\bibitem[HLTW19]{HARUTYUNYAN2019}
Ararat Harutyunyan, Tien-Nam Le, Stéphan Thomassé, and Hehui Wu.
\newblock Coloring tournaments: From local to global.
\newblock {\em Journal of Combinatorial Theory, Series B}, 2019.

\bibitem[HM97]{survey1}
John~F. {Hughes} and Gary {MacGillivray}.
\newblock {The achromatic number of graphs: A survey and some new results}.
\newblock {\em {Bull Inst Combin Appl}}, 19:27--56, 1997.

\bibitem[KO04]{kuehn}
Daniela K{\"u}hn and Deryk Osthus.
\newblock Induced subdivisions in ks,s-free graphs of large average degree.
\newblock {\em Combinatorica}, 24(2):287--304, Apr 2004.

\bibitem[LM17]{dig4}
Zhentao {Li} and Bojan {Mohar}.
\newblock Planar digraphs of digirth four are 2-colorable.
\newblock {\em SIAM J. Discrete Math.}, 31:2201--2205, 2017.

\bibitem[{Mad}67]{mader67}
Wolfgang {Mader}.
\newblock {Homomorphieeigenschaften und mittlere Kantendichte von Graphen}.
\newblock {\em {Mathematische Annalen}}, 174 (4):265--268, 1967.

\bibitem[MW16]{fractionalNL}
Bojan {Mohar} and Hehui {Wu}.
\newblock Dichromatic number and fractional chromatic number.
\newblock {\em Forum of Mathematics, Sigma,}, 4, E32, 2016.

\bibitem[RRST96]{Reed1996}
Bruce Reed, Neil Robertson, Paul Seymour, and Robin Thomas.
\newblock Packing directed circuits.
\newblock {\em Combinatorica}, 16(4):535--554, Dec 1996.

\bibitem[{Sop}14]{oriented}
\'Eric {Sopena}.
\newblock {Complete oriented colourings and the oriented achromatic number.}
\newblock {\em {Discrete Appl. Math.}}, 173:102--112, 2014.

\end{thebibliography}
\bibliographystyle{alpha}

\end{document}